\documentclass[12pt]{amsart}

\usepackage{graphicx, amssymb, mathrsfs}
\usepackage{enumerate}

\newtheorem{theorem}{Theorem}[section]

\newtheorem{lemma}[theorem]{Lemma}

\newtheorem*{ntheorem}{\bf Main Theorem}
\newtheorem*{ncorollary}{\bf Corollary}

\theoremstyle{definition}

\newcommand{\C}{\mathbb{C}}
\newcommand{\D}{\mathbb{D}}

\newcommand{\CC}{\mathscr{C}}

\newcommand{\N}{\mathbb{N}}

\newcommand{\id}{\operatorname{id}}

\DeclareMathOperator{\diam}{diam}
\DeclareMathOperator{\dist}{dist}

\title{A Rudin-Carleson theorem with uniform approximation for manifold-valued maps}

\author{Benedikt Steinar Magnússon}
\author{Tyson Ritter}
\address{B. Magnússon, Science Institute, University of Iceland, Iceland} 
\email{bsm@hi.is}
\address{T. Ritter, Department of Mathematics and Physics, University of Stavanger, Norway} 
\email{tyson.ritter@uis.no}

\subjclass[2020]{32H02 (Primary), 30E10, 30E25, 32Q56 (Secondary)}

\keywords{Rudin-Carleson interpolation, Runge approximation, holomorphic interpolation, Oka manifold, analytic disc.}

\begin{document}

\begin{abstract}
	Given a closed set $E \subset \partial {\mathbb D}$ of measure zero and a continuous function $\varphi : E \to {\mathbb C}$, the classical Rudin-Carleson interpolation theorem states that there exists a continuous function $F : \overline {\mathbb D} \to {\mathbb C}$ that is holomorphic on ${\mathbb D}$ and satisfies $F\rvert_E = \varphi$. In this paper we obtain a generalisation of the Rudin-Carleson theorem for maps $\varphi : E \to X$ into arbitrary connected complex manifolds $X$ that combines interpolation of $\varphi$ on $E$ with uniform approximation on compact subsets of $\overline {\mathbb D} \setminus E$ of another given continuous map $f:\overline {\mathbb D} \to X$ that is holomorphic on ${\mathbb D}$. Under the further assumption that $X$ is an Oka manifold we obtain a corollary that combines Rudin-Carleson interpolation of a continuous map $\varphi : \overline {\mathbb D} \to X$ on $E$ with Runge approximation of $\varphi$ on a compact set $K\subset {\mathbb D}$ without any holes on which $\varphi$ is holomorphic.
\end{abstract}

\maketitle

\section{Introduction and main results}

\noindent
Throughout this paper we denote by $S$ the unit circle $\partial \D \subset \C$. We begin by recalling the classical Rudin-Carleson theorem.

\begin{theorem}[Rudin-Carleson \cite{Carleson1957, Rudin1956}]
	\label{thm_Rudin}
	Let $E \subset S$ be a closed set of measure zero and let $\varphi : E \to \C$ be a continuous function. Suppose that $\varphi(E) \subset T$, where $T \subset \C$ is homeomorphic to $\overline \D$. Then, there exists a continuous function $f : \overline \D \to \C$ that is holomorphic on $\D$ such that $f(z) = \varphi(z)$ for all $z \in E$ and which satisfies $f(\overline \D) \subset T$.
\end{theorem}

Since its publication, this result has received various extensions and generalisations in the literature. To mention a few, Bishop \cite{Bishop1962} gave a generalisation of this result to closed subspaces of continuous complex-valued functions on compact Hausdorff spaces. Later, in a series of papers starting with \cite{Globevnik1975}, Globevnik studied various generalisations of the Rudin-Carleson theorem in the case of maps into complex Banach spaces. More recently, Izzo \cite{Izzo2018} has shown that it is possible to combine Rudin-Carleson interpolation with Pick interpolation at finitely many interior points of the unit disc. Finally, using different techniques than those in the current paper Brudnyi \cite{Brudnyi2022} obtained a version of the Rudin-Carleson theorem for manifold-valued maps, but his results do not include uniform approximation on compact subsets of $\overline \D \setminus E$, in contrast with our main result.

Our main theorem is as follows. It appears to be the first instance in which Rudin-Carleson interpolation has been combined with uniform approximation.

\begin{ntheorem}[Theorem \ref{thm_main_generaltarget}]
	Let $X$ be a connected complex manifold equipped with a distance function induced by a complete Riemannian metric.	Let $f : \overline \D \to X$ be a continuous map that is holomorphic on $\D$. Let $E \subset S$ be a closed set of measure zero and let $\varphi : E \to X$ be a continuous map. Let $c \in \D$. Then, $f$ can be uniformly approximated on compact subsets of $\overline \D \setminus E$ by continuous maps $F:\overline \D \to X$ that are holomorphic on $\D$, such that $F\vert_E = \varphi$ and $F(c) = f(c)$. Furthermore, if $U \subset X$ is any connected neighbourhood of $f(\overline \D) \cup \varphi(E)$ then it is possible to ensure that $F(\overline \D) \subset U$.
\end{ntheorem}

Recall that a complex manifold $X$ is said to be \emph{Oka} if it satisfies any of a number of equivalent properties, each stating in some form that there is a rich supply of holmorphic maps from any Stein manifold $Y$ into $X$, in the sense that there exist only topological obstructions to finding holomorphic solutions to approximation and interpolation problems concerning maps $Y \to X$. We refer the reader to the recent paper \cite{Forstneric2025} for an introduction to Oka manifolds and Oka theory more generally (see also the monograph \cite{Forstneric2017}). In this paper we define a complex manifold $X$ to be Oka if it satisfies the \emph{basic Oka property with approximation and interpolation}, as follows. Let $Y$ be a Stein manifold, $K \subset Y$ be a holomorphically convex compact subset, and $L \subset Y$ be a closed complex submanifold. Then, every continuous map $f : Y \to X$ which is holomorphic on a neighbourhood of $K$ and holomorphic on $L$ can be uniformly approximated on $K$ and interpolated on $L$ by a holomorphic map $F : Y \to X$.

In the case that $X$ is an Oka manifold our main theorem yields the following corollary.

\begin{ncorollary}
	Let $X$ be an Oka manifold. Let $E \subset S$ be a closed set of measure zero and $K \subset \D$ be a compact set without any holes (that is, such that $\D \setminus K$ has no relatively compact connected components, where the closure is taken with respect to $\D$). Let $c \in \D$. Given a continuous map $\varphi : \overline \D \to X$ that is holomorphic on a neighbourhood of $K$, there exists a continuous map $F : \overline \D \to X$, holomorphic on $\D$, that uniformly approximates $\varphi$ on $K$, interpolates $\varphi$ on $E$, and satisfies $F(c) = \varphi(c)$.
\end{ncorollary}



\begin{proof}
	Begin by extending $\varphi$ continuously from $\overline \D$ to all of $\C$. Since $X$ is Oka, applying the basic Oka property with approximation and interpolation to $\varphi$ yields an entire map $f : \C \to X$ that uniformly approximates $\varphi$ on $K$ and satisfies $f(c) = \varphi(c)$. Applying the Main Theorem to the maps $f\rvert_{\overline \D} : \overline \D \to X$ and $\varphi\rvert_E : E \to X$ yields the result.
\end{proof} 

\section{A classical Rudin-Carleson theorem for maps close to the identity}

\noindent 
In this section we establish a special case of the classical Rudin-Carleson theorem for maps close to the identity, the first of two essential ingredients required in the proof of our main theorem.

In the statement of the classical Rudin-Carleson theorem (Theorem \ref{thm_Rudin}), suppose that the continuous function $\varphi : E \to \C$ has image in the unit circle $S$, $\varphi(E) \subset S$. It then follows that the interpolating function $f : \overline \D \to \C$ provided by the theorem can be chosen to satisfy $f(\overline \D) \subset \overline \D$. Further suppose now that $\varphi$ is uniformly close to the identity  on $E$. The following result shows that, in this case, $f$ can be chosen such that it is both uniformly close to the identity on all of $\overline \D$ and fixes the origin.

\begin{theorem}
	\label{thm_RudinCarleson}
	Let $E \subset S$ be a closed set of measure zero. Let $\epsilon > 0$ and let $\varphi : E \to \C$ be a continuous function such that $\varphi(E) \subset S$ and $\lvert \varphi(z) - z \rvert < \epsilon$ for all $z \in E$. Then, there exists a continuous function $f : \overline \D \to \C$ that is holomorphic on $\D$ such that $f(\overline \D) \subset \overline \D$, $\lvert f(z) - z \rvert < \epsilon$ for all $z \in \overline \D$ and $f(z) = \varphi(z)$ for all $z \in E$. Furthermore, $f(0) = 0$.
\end{theorem}
\begin{proof}
	Define the continuous function $\tilde{\varphi} : E \to \C$ by $\tilde{\varphi}(z) = \varphi(z)/z$ for $z \in E$. We have $\lvert \tilde{\varphi}(z) \vert = \lvert \varphi(z)\rvert / \lvert z \rvert = 1$ for all $z \in E$. We also have, for $z \in E$,
	\[ \lvert \tilde\varphi(z) - 1 \rvert = \left\lvert \frac{\varphi(z)}{z} - 1\right\rvert = \left\lvert \frac{\varphi(z) - z}{z} \right\rvert < \epsilon\,,\]
	so that $\tilde{\varphi}(E) \subset B(1; \epsilon)$, the ball of radius $\epsilon$ centred at the point $1$.
	Since $E$ is compact, $\tilde{\varphi}(E)$ is contained in a compact arc $A \subset S \cap B(1; \epsilon)$. Let $T \subset \overline \D \cap B(1; \epsilon)$ be homeomorphic to $\overline \D$ such that $A\subset T$, so that we also have $\tilde{\varphi}(E) \subset T$.
	
	By Theorem \ref{thm_Rudin} there exists a continuous function $\tilde{f} : \overline \D \to \C$ that is holomorphic on $\D$ such that $\tilde{f}(z) = \tilde{\varphi}(z)$ for all $z \in E$, with $\tilde{f}(\overline \D) \subset T$. Define $f : \overline \D \to \C$ by $f(z) = z \tilde{f}(z)$. Then, for $z \in E$ we have $f(z) = z \tilde{f}(z) = z \tilde{\varphi}(z) = \varphi(z)$. Furthermore, for all $z \in \overline\D$ we have $\lvert f(z) \rvert = \lvert z \rvert \lvert \tilde{f}(z)\rvert \le \lvert \tilde{f}(z) \rvert \le 1$, since $\tilde{f}(z) \in T \subset \overline \D$. For $z \in \overline \D$ we also have
	\[\lvert f(z) - z \rvert = \lvert z \tilde{f}(z) - z \rvert = \lvert z \rvert \lvert \tilde{f}(z) - 1 \rvert \le \lvert \tilde{f}(z) - 1 \rvert < \epsilon\]
	since $\tilde{f}(z) \in T \subset B(1;\epsilon)$. Finally, $f(0) = 0 \tilde{f}(0) = 0$.
\end{proof}


We now show that instead of choosing $f$ such that it fixes the origin, we may instead choose $f$ such that it fixes any other single point in $\D$. We can of course also choose $f$ such that it in addition fixes pointwise any set of measure zero $M \subset S \setminus E$ by extending $\varphi$ to $M$ by the identity and redefining $E$ as $M \cup E$. Note that by the Schwarz-Pick lemma it is not possible for $f$ to fix more than a single point in $\D$ unless $\varphi = \id_E$.

\begin{theorem}
	\label{thm_RudinCarleson_general_c}
	Let $E \subset S$ be a closed set of measure zero. Let $\epsilon > 0$ and let $c \in \D$. Then, there exists $\delta > 0$ such that for every continuous function $\varphi : E \to \C$ with $\varphi(E) \subset S$ and $\lvert \varphi(z) - z \rvert < \delta$ for all $z \in E$, there is a continuous function $g : \overline \D \to \C$ that is holomorphic on $\D$ such that $g(\overline \D) \subset \overline \D$, $\lvert g(z) - z \rvert < \epsilon$ for all $z \in \overline \D$, $g(z) = \varphi(z)$ for all $z \in E$, and $g(c) = c$.
\end{theorem}
\begin{proof}
	Define $\alpha : \overline \D \to \C$ by $\alpha(z) = (z-c)/(1-\overline c z)$.	Then, $\alpha$ is a homeomorphism of $\overline \D$ onto itself (necessarily with $\alpha(S) = S$) which restricts to a holomorphic automorphism of $\D$, and $\alpha(c) = 0$. By the uniform continuity of $\alpha^{-1}$ on $\overline{\D}$ there exists $\eta > 0$ such that for $x, y \in \overline \D$,
	\begin{equation}
		\label{unif_cont_alphainv}
		\lvert x - y \rvert < \eta \implies \lvert \alpha^{-1}(x) - \alpha^{-1}(y) \rvert < \epsilon\,.
	\end{equation}
	The uniform continuity of $\alpha$ on $\overline{\D}$ then gives $\delta > 0$ such that for $x, y \in \overline \D$,
	\begin{equation}
		\label{unif_cont_alpha}
		\lvert x - y \rvert < \delta \implies \lvert \alpha(x) - \alpha(y) \rvert < \eta\,.
	\end{equation}
	
	Suppose that $\varphi$ is given as stated in the theorem. Defining $\psi : E \to \C$ by $\psi = \alpha \circ \varphi$, we have that $\psi(E) \subset S$. (Note that since $\varphi$ is close to the identity on $E$, $\psi$ is close to $\alpha$.) Defining $\tilde{\psi} : E \to \C$ by $\tilde{\psi}(z) = \psi(z)/\alpha(z)$, we also have that $\tilde{\psi}(E) \subset S$. For $z \in E$ we have
	
	\[\left\lvert \tilde{\psi}(z) - 1 \right\rvert = \left\lvert \frac{\alpha(\varphi(z))-\alpha(z)}{\alpha(z)} \right\rvert = \left\lvert \alpha(\varphi(z))-\alpha(z)\right\rvert < \eta\,,\]
	
	\noindent where we have used (\ref{unif_cont_alpha}) and the assumption that $\lvert \varphi(z) - z \rvert < \delta$ for $z \in E$. This shows that $\tilde{\psi}(E) \subset B(1;\eta)$. As in the proof of Theorem \ref{thm_RudinCarleson}, we obtain via Theorem \ref{thm_Rudin} a continuous function $\tilde{f}: \overline \D \to \C$ that is holomorphic on $\D$, which extends $\tilde{\psi}$ from $E$, and which satisfies $\tilde{f}(\overline \D) \subset \overline \D \cap B(1;\eta)$.
	
	Let $f : \overline \D \to \C$ be defined by $f(z) = \alpha(z) \tilde{f}(z)$. Then for $z \in E$ we have $f(z) = \psi(z) = \alpha (\varphi(z))$. We have $f(\overline \D)\subset \overline \D$ and $f(c) = 0$. For $z \in \overline \D$ we have $\lvert f(z) - \alpha(z)\rvert < \eta$.
	
	To complete the proof, define $g : \overline \D \to \C$ by $g = \alpha^{-1}\circ f$. For $z \in E$ we have $g(z) = \varphi(z)$. We have $g(\overline \D) \subset \overline \D$ and $g(c) = c$. Finally, for $z \in \overline \D$ we have, by (\ref{unif_cont_alphainv}),
	
	\[\lvert g(z) - z \rvert = \lvert \alpha^{-1}(f(z)) - \alpha^{-1}(\alpha(z))\rvert < \epsilon\,.\]	\end{proof}
	
\section{Extending fingers using an approach of Forsterni\v c-Wold}

\noindent
In this section we prove Theorem \ref{thm_ForstnericWold} on modifying maps $\overline \D \to \C^n$ in a neighbourhood of finitely many boundary points by attaching fingers to the image along arcs in $\C^n$, while keeping the map approximately fixed elsewhere and fixing (precisely) finitely many points in $\D$. Our method follows closely the technique for exposing boundary points developed by Forstneri\v c and Wold in \cite{ForstnericWold2009} as part of their work on constructing proper holomorphic embeddings into $\C^2$ of interiors of compact bordered Riemann surfaces already holomorphically embedded in $\C^2$. In their work it was important that the process of attaching fingers, or exposing boundary points, preserved the injectivity of the given map near the boundary. We do not, however, require this condition. Another difference is that our maps $\overline \D \to \C^n$ are only continuous on the boundary, whereas those considered in \cite{ForstnericWold2009} are $\CC^1$.





The following is a slightly strengthened version of Lemma 2.2 in Forstneri\v c-Wold. An examination of their proof shows that the stated condition below holds regarding control on the images of the neighbourhoods of the boundary points.

\begin{lemma}
	\label{lem_attachspike}
	Let $R$ be a Riemann surface and $D$ be a relatively compact smoothly bounded domain with nonempty boundary in $R$, not necessarily connected. Given pairwise distinct points $a_1,a_2,\dots,a_k \in \overline D$ with $a_1 \in \partial D$, neighbourhoods $U \Subset U' \subset R$ of $a_1$, a point $b \in R \setminus \overline D$ in the same connected component of $R \setminus D$ as $a_1$, and a positive integer $N \in \N$, there is a smooth diffeomorphism $\phi : \overline D \to \phi(\overline D) \subset R$ satisfying the following:
	\begin{enumerate}[(i)]
		\item $\phi : D \to \phi(D)$ is a biholomorphism,
		\item $\phi(a_1) = b$,
		\item $\phi$ is tangent to the identity to order $N$ at each of the points $a_2,\dots,a_k$,
		\item $\phi$ is as close as desired to the identity map on $\overline D \setminus U$ in the smooth topology on the space of maps.
	\end{enumerate}
	Furthermore, if $\gamma \subset R$ is a smooth Jordan arc with endpoints $a_1$ and $b$ such that $\gamma \cap \overline D = \{ a_1 \}$ and the tangent lines to $\gamma$ and $\partial D$ at the point $a_1$ are transverse, and $V$ is a neighbourhood of $\gamma$, then it is possible to ensure that $\phi(\overline U \cap \overline D) \subset U' \cup V$.
\end{lemma}

An inductive argument using Lemma \ref{lem_attachspike} then gives the following result (compare to Forstneri\v c-Wold, Theorem 2.3).
\begin{theorem}
	\label{thm_spikes}
	Let $D$ be a relatively compact smoothly bounded domain in a Riemann surface $R$.
	Choose finitely many pairwise distinct points $a_1,\dots,a_k \in \partial D$, $b_1,\dots,b_k \in R \setminus \overline D$, and $c_1, \dots, c_l \in \overline D \setminus \{a_1,\dots,a_k\}$ such that for each $j = 1,\dots,k$ the points $a_j$ and $b_j$ belong to the same connected component of $R\setminus D$. Let $\gamma_1,\dots,\gamma_k$ be pairwise disjoint smooth Jordan arcs such that, for each $j = 1,\dots,k$, $\gamma_j$ has endpoints $a_j$ and $b_j$, $\gamma_j \cap \overline D = \{a_j\}$, and the tangent lines to $\gamma_j$ and $\partial D$ at the point $a_j$ are transverse. For $j = 1,\dots,k$, let $V_j \subset R$ be a neighbourhood of $\gamma_j$ such that $V_1,\dots,V_k$ are pairwise disjoint, and let $U_j \Subset V_j$ be neighbourhoods of the points $a_j$.
	
	Then, for all $N \in \N$ there exists a smooth diffeomorphism $\phi : \overline D \to \phi(\overline D)$ onto a smoothly bounded domain $\phi(\overline D) \subset R$ such that:
	\begin{enumerate}[(i)]
		\item $\phi : D \to \phi(D)$ is a biholomorphism,
		\item $\phi(a_j) = b_j$ for $j = 1,\dots,k$,
		\item $\phi(\overline U_j \cap \overline D) \subset V_j$ for $j = 1,\dots,k$,
		\item $\phi$ is tangent to the identity map to order $N$ at $c_i$ for $i = 1,\dots,l$.		
	\end{enumerate}
	Furthermore, $\phi$ can be chosen as close as desired to the identity map in the smooth topology on $\overline D \setminus \bigcup_{j = 1}^k U_j$.
\end{theorem}


\begin{proof}
	Without loss of generality, assume that $N \ge 2$. For the base case, choose a neighbourhood $U'_1$ of $a_1$ such that $U_1 \Subset U'_1 \Subset V_1$, and	let $\phi_1 : \overline D \to \phi_1(\overline D)$ be a diffeomorphism given by Lemma \ref{lem_attachspike} such that:
	\begin{itemize}
		\item $\phi_1 : D \to \phi_1(D)$ is a biholomorphism,
		\item $\phi_1(a_1) = b_1$,
		\item $\phi_1(\overline U_1 \cap \overline D) \subset U'_1 \cup V_1 \subset V_1$,
		\item $\phi_1$ is tangent to the identity to order $N$ at the points $a_j$ for $j = 2,\dots,k$ and the points $c_i$ for $i = 1,\dots,l$,
		\item $\phi_1$ is uniformly close to the identity map in the smooth topology on $\overline D \setminus U_1$.
	\end{itemize}
	We may also assume that for each $j = 2,\dots,k$ we have $\phi_1(U_j) \Subset V_j$, that the arc $\gamma_j$ meets $\partial \phi_1(D)$ transversally at $\phi_1(a_j) = a_j$, and $\gamma_j \cap \phi_1(\overline D) = \{a_j\}$.
	
	In the second step, we choose a neighbourhood $U'_2$ of $a_2$ such that $\phi_1(U_2) \Subset U'_2 \Subset V_2$, and obtain a diffeomorphism $\phi_2 : \phi_1(\overline D) \to \phi_2(\phi_1(\overline D))$ such that:
	\begin{itemize}
		\item $\phi_2 : \phi_1(D) \to \phi_2(\phi_1(D))$ is a biholomorphism,
		\item $\phi_2(a_2) = b_2$,
		\item $\phi_2(\phi_1(\overline U_2) \cap \phi_1(\overline D)) \subset U'_2 \cup V_2 \subset V_2$,
		\item $\phi_2$ is tangent to the identity to order $N$ at the point $b_1$, the points $a_j$ for $j = 3,\dots,k$ and the points $c_i$ for $i = 1,\dots,l$,
		\item $\phi_2$ is uniformly close to the identity map in the smooth topology on $\phi_1(\overline D) \setminus \phi_1(U_2)$.
	\end{itemize}
	We may also assume that $\phi_2(\phi_1(\overline U_1 \cap \overline D)) \subset V_1$ and that for each $j = 3,\dots,k$ we have $\phi_2(\phi_1(U_j)) \Subset V_j$, that each arc $\gamma_j$ meets $\partial \phi_2(\phi_1(D))$ transversally at $a_j$, and $\gamma_j \cap \phi_2(\phi_1(\overline D)) = \{a_j\}$.
	
	Continuing in this manner, we let $\phi = \phi_k \circ \phi_{k-1} \circ \dots \circ \phi_1$. Then, $\phi : \overline D \to \phi(\overline D)$ is a diffeomorphism satisfying the necessary properties.	
\end{proof}	

We are now in a position to prove the required result on attaching fingers to the image of continuous maps $\overline \D \to \C^n$ that are holomorphic on $\D$, the second essential ingredient in the proof of our main theorem.

\begin{theorem}
	\label{thm_ForstnericWold}
		Let $f : \overline \D \to \C^n$ $(n \ge 1)$ be a continuous map that is holomorphic on $\D$. Let $a_1, \dots, a_k \in \partial \D$ be distinct points. For $j = 1, \dots, k$, let $U_j \subset \overline \D$ be a (relatively open) neighbourhood of $a_j$ such that $U_1,\dots,U_k$ are pairwise disjoint. Let $b_1, \dots, b_k \in \C^n$. For $j = 1, \dots, k$, let $\lambda_j \subset \C^n$ be a continuous arc from $f(a_j)$ to $b_j$. Let $c_1, \dots, c_s \in \overline \D \setminus \{a_1,\dots,a_k\}$. For $j = 1, \dots, k$, let $V_j \subset \C^n$ be a neighbourhood of $f(\overline U_j) \cup \lambda_j$, and let $V \subset \C^n$ be a neighbourhood of $f(\overline \D) \cup \bigcup_{j=1}^k \lambda_j$. Then, $f$ can be uniformly approximated on $\overline \D \setminus \bigcup_{j=1}^k U_j$ by continuous functions $g : \overline \D \to \C^n$ that are holomorphic on $\D$, such that $g(a_j) = b_j$ for $j = 1 , \dots, k$, $g(U_j) \subset V_j$ for $j = 1 , \dots, k$,
		$g(c_i) = f(c_i)$ for $i = 1,\dots,s$, and $g(\overline \D) \subset V$.
\end{theorem}


\begin{proof}
	
	Let $\epsilon > 0$. Fix $r > 1$ and let $L_j \subset \C$ be the line segment from $a_j$ to $ra_j$, $j = 1,\dots,k$. Let $K = \overline \D \cup \bigcup_{j=1}^kL_j$ and define $f' : K \to \C$ by extending $f$ continuously to each of the $L_j$ such that $L_j$ maps to $\lambda_j$ with $f'(a_j) = f(a_j)$ and $f'(ra_j) = b_j$.
	
	By Mergelyan's theorem (with interpolation) we can approximate $f'$ by an entire function $h : \C \to \C^n$ such that, for $j = 1,\dots,k$ and $i = 1,\dots,s$:
	
	\begin{enumerate}[(i)]
		\item $\lVert h - f' \rVert_{K} < \epsilon/2$,
		\item $h(ra_j) = f'(ra_j) = b_j$,
		\item $h(c_i) = f'(c_i) = f(c_i)$,
		\item $h(\overline U_j \cup L_j) \subset V_j$,
		\item $h(K) \subset V$.
	\end{enumerate}
	
	Restrict $h$ to a relatively compact neighbourhood $W$ of $K$ such that $h(W) \subset V$. For $j = 1,\dots,k$, let $W_j \subset W$ be a neighbourhood of $\overline U_j \cup L_j$ such that $h(\overline W_j) \subset V_j$. By the uniform continuity of $h$ on $W$ there exists $\delta > 0$ such that, for $z,w \in W$, $\lvert z - w \rvert < \delta$ implies $\lvert h(z) - h(w)\lvert < \epsilon/2$.
	
	By Theorem \ref{thm_spikes} there exists a smooth diffeomorphism $\sigma :\overline \D \to \sigma(\overline \D)$ which is holomorphic on $\D$ and satisfies
	\begin{itemize}
		\item $\sigma(\overline \D) \subset W$,
		\item $\lVert \sigma - \id \rVert_{\overline\D\setminus \bigcup_{j = 1}^k U_j} < \delta$,
		\item $\sigma(a_j) = ra_j$ for $j = 1,\dots,k$,
		\item $\sigma(c_i) = c_i$,
		\item $\sigma(\overline U_j) \subset W_j$.
	\end{itemize}
	
	Note: Strictly speaking, Theorem \ref{thm_spikes} requires the neighbourhoods $U_j$ to be open subsets of $W$, but it is clear that given a relatively open set $U_j \subset \overline \D$ satisfying $\overline U_j \subset W_j$, there exists an open set $U'_j \subset \C$ such that $U_j = U'_j \cap \overline \D$ and $\overline U'_j \subset W_j$.
	
	Define $g : \overline \D \to \C^n$ by $g = h \circ \sigma$. Then for $z \in \overline\D\setminus \bigcup_{j = 1}^k U_j$ we have
	\begin{equation*}
		\lvert g(z) - f(z) \rvert = \lvert h(\sigma(z)) - f(z)\rvert \le \lvert h(\sigma(z)) - h(z) \rvert + \lvert h(z) - f(z)\rvert < \epsilon\,.
	\end{equation*}
	It is immediate to check that $g$ satisfies the remaining conditions as required.
\end{proof}

\section{Main result for maps into $\C^n$}

\noindent
We now establish a version of our main theorem for maps into $\C^n$. In the following section we state and prove a corresponding result for maps into arbitrary connected complex manifolds.

\begin{theorem}
	\label{thm_main_Cn}
	Let $f : \overline \D \to \C^n$ $(n \ge 1)$ be a continuous map that is holomorphic on $\D$. Let $E \subset S$ be a closed set of measure zero and let $\varphi : E \to \C^n$ be a continuous map. Let $c \in \D$. Then, $f$ can be uniformly approximated on compact subsets of $\overline \D \setminus E$ by continuous maps $F:\overline \D \to \C^n$ that are holomorphic on $\D$, such that $F\vert_E = \varphi$ and $F(c) = f(c)$. Furthermore, if $U \subset \C^n$ is any connected neighbourhood of $f(\overline \D) \cup \varphi(E)$ then it is possible to ensure that $F(\overline \D) \subset U$.
\end{theorem}

\begin{proof}
	Suppose that $c = 0$ (the general case for arbitrary $c \in \D$ is handled similarly, using Theorem \ref{thm_RudinCarleson_general_c} in place of Theorem \ref{thm_RudinCarleson}). Let $K \subset \overline{\D} \setminus E$ be a compact set and let $\epsilon > 0$. Let $V \Subset U \subset \C^n$ be a connected neighbourhood of $f(\overline \D) \cup \varphi(E)$. Define a sequence $\eta_j > 0$, $j = 1,2,\dots$, by $\eta_j = \epsilon/2^{j+3}$, shrinking $\epsilon$ first if necessary such that for all $z \in E$, the ball $B(\varphi(z); \eta_1)$ of radius $\eta_1$ centred at $\varphi(z)$ is contained in $V$.
	
	Let $f_0 = f$. By induction we will construct for each $j > 0$ a continuous map $f_j : \overline{\D} \to \C^n$ that is holomorphic on $\D$, a partition $(E_i^{(j)})_{i=1}^{k_j}$ of $E$ into pairwise disjoint closed subsets, and points $(a_i^{(j)})_{i=1}^{k_j}$ with $a_i^{(j)} \in E_i^{(j)}$, $i = 1,\dots,k_j$, such that the following conditions are satisfied.
	
	\begin{enumerate}[(a)]
		\item \label{conv_E} $\lVert f_j - \varphi \rVert_E < \eta_j$ for all $j > 0$,
		\item \label{conv_K} $\lVert f_j - f_{j-1} \rVert_K < \epsilon/2^j$ for all $j > 0$,
		\item \label{conv_D} $\lVert f_j - f_{j-1} \rVert_{\overline{\D}} < \epsilon/2^j$ for all $j > 1$,
		\item \label{fix_point} $f_j(0) = f(0)$ for all $j \ge 0$,
		\item \label{image_E} $f_j(E_i^{(j)}) = \varphi(a_i^{(j)})$ for $i = 1, \dots, k_j$, for all $j > 0$,
		\item \label{control_image} $f_j(\overline{\D}) \subset V$ for all $j > 0$.
	\end{enumerate}
	
	
	Suppose that such a construction has been completed. It follows from (\ref{conv_D}) that $(f_j)_{j = 1}^\infty$ is a uniformly Cauchy sequence on $\overline{\D}$. Defining $F := \lim\limits_{j \to \infty} f_j$, we therefore have that $F : \overline{\D} \to \C^n$ is a continuous map whose restriction to $\D$ is holomorphic, and $F(0) = f(0)$ by (\ref{fix_point}). By (\ref{conv_E}) we have that $F\vert_E = \varphi$. From (\ref{conv_K}) it follows that $\lVert F - f \rVert_K \le \epsilon$. Finally, (\ref{control_image}) gives that $F(\overline{\D}) \subset \overline{V} \subset U$.
	
	Let $j = 1$. We construct a continuous map $f_1 : \overline{\D} \to \C^n$ that is holomorphic on $\D$, a partition $(E_i^{(1)})_{i=1}^{k_1}$ of $E$ into closed subsets, and points $(a_i^{(1)})_{i=1}^{k_1}$ such that $a_i^{(1)} \in E_i^{(1)}$ for $i = 1,\dots,k_1$, satisfying conditions (\ref{conv_E}), (\ref{conv_K}), (\ref{fix_point}), (\ref{image_E}), and (\ref{control_image}). (Condition (\ref{conv_D}) does not apply when $j = 1$.)
		
	By the uniform continuity of $f_0$ on $\overline{\D}$ there exists $\delta_1$ such that, for $x, y \in \overline{\D}$,
	\begin{equation}
		\label{unif_cont_f_0}
		\lvert x - y \rvert < \delta_1 \implies \lvert f_0(x) - f_0(y) \rvert < \epsilon/4\,.
	\end{equation}
	Shrinking $\delta_1$, if necessary, we may further assume by the uniform continuity of $\varphi$ on $E$ that, for $x, y \in E$,
	\begin{equation}
		\label{unif_cont_varphi_eta_1}
		\lvert x - y \rvert < \delta_1 \implies \lvert \varphi(x) - \varphi(y) \rvert < \eta_1\,.
	\end{equation}
	We also ensure that $\delta_1$ is chosen sufficiently small so that $\delta_1 < \dist(K, E)$.
	
	
	Partition the set $E$ into finitely many pairwise disjoint closed sets $E^{(1)}_1, \dots, E^{(1)}_{k_1}$ such that $\diam(E^{(1)}_i) < \delta_1$ for $i = 1,\dots,k_1$. For each $i = 1,\dots,k_1$ we choose a point $a^{(1)}_i \in E^{(1)}_i$. Defining $\alpha_1 : E \to \partial \D$ by $\alpha_1(E^{(1)}_i) = a^{(1)}_i$, $i = 1,\dots,k_1$, we have $\lVert \alpha_1 - \id \rVert_E < \delta_1$. By Theorem \ref{thm_RudinCarleson}, $\alpha_1$ extends to a continuous map $\alpha_1 : \overline{\D} \to \overline{\D}$ (which we still denote by $\alpha_1$), holomorphic on $\D$, such that $\lVert \alpha_1 - \id \rVert_{\overline{\D}} < \delta_1$ and $\alpha_1(0) = 0$. Without loss of generality we may assume that $\alpha_1$ is not constant. It then follows by the open mapping theorem that $\alpha_1(\D) \subset \D$. From the fact that $\delta_1 < \dist(K, E)$ we also see that $\alpha_1(K)$ is a compact subset of $\overline{\D}\setminus E$.
	
	For each $i = 1, \dots, k_1$, let $U^{(1)}_i \subset \overline{\D}$ be a (relatively) open neighbourhood of the point $a^{(1)}_i \in E \subset \partial \D$ such that $U^{(1)}_i \cap \alpha_1(K) = \varnothing$, and let $\gamma^{(1)}_i \subset V$ be an arc starting at $f_0(a^{(1)}_i) = f(a^{(1)}_i)$ and ending at $\varphi(a^{(1)}_i)$. By Theorem \ref{thm_ForstnericWold} there exists a continuous map $g_1 : \overline{\D} \to \C^n$, holomorphic on $\D$, with the following properties.
	
	\begin{itemize}
		\item $\lVert f_0 - g_1 \rVert_{\overline{\D} \setminus \bigcup_{i = 1}^{k_1} U^{(1)}_i} < \epsilon / 4$.
		\item $g_1(a^{(1)}_i) = \varphi(a^{(1)}_i)$ for $i = 1,\dots,k_1$.
		\item $g_1(0) = f_0(0) = f(0)$.
		\item $g_1(\overline{\D}) \subset V$.
	\end{itemize}
	
	Let $f_1 = g_1 \circ \alpha_1$. Then, $f_1 : \overline{\D} \to \C^n$ is a continuous map whose restriction to $\D$ is holomorphic, and $f_1$ satisfies properties (\ref{conv_E}), (\ref{conv_K}), (\ref{fix_point}), (\ref{image_E}), and (\ref{control_image}). Indeed, we have $f_1(\overline{\D}) = g_1(\alpha_1(\overline{\D})) \subset g_1(\overline{\D}) \subset V$. We also have $f_1(0) = g_1(\alpha_1(0)) = g_1(0) = f(0)$. Let $z \in E$, then $z \in E^{(1)}_i$ for some $i \in \{1,\dots,k_1\}$. We then have $f_1(z) = g_1(\alpha_1(z)) = g_1(a^{(1)}_i) = \varphi(a^{(1)}_i)$. But then (\ref{unif_cont_varphi_eta_1}) gives $\lvert f_1(z) - \varphi(z) \rvert = \lvert \varphi(a^{(1)}_i) - \varphi(z) \rvert < \eta_1$, since $a^{(1)}_i, z \in E^{(1)}_i$ and $\diam(E^{(1)}_i) < \delta_1$. Let $z \in K$, then
	\begin{eqnarray*}
		\lvert f_1(z) - f_0(z) \rvert & = & \lvert g_1(\alpha_1(z)) - f_0(z) \rvert \\
		& \le & \lvert g_1(\alpha_1(z)) - f_0(\alpha_1(z)) \rvert + \lvert f_0(\alpha_1(z)) - f_0(z) \rvert \\
		& < & \epsilon / 4 + \epsilon / 4 = \epsilon / 2
	\end{eqnarray*}
	where we have used that $\alpha_1(z) \in \alpha_1(K) \subset \overline{\D} \setminus \bigcup_{i = 1}^{k_1} U^{(1)}_i$, and (\ref{unif_cont_f_0}) together with $\lvert \alpha_1(z) - z \rvert < \delta_1$. This completes the first step of the induction.
	
	Suppose now that for some $j > 0$ we have a continuous map $f_j : \overline{\D} \to \C^n$ that is holomorphic on $\D$, a partition $(E_i^{(j)})_{i=1}^{k_j}$ of $E$ into closed subsets, and points $(a_i^{(j)})_{i=1}^{k_j}$ with $a_i^{(j)} \in E_i^{(j)}$ for $i = 1,\dots,k_j$, such that conditions (\ref{conv_E})--(\ref{control_image}) are satisfied (excluding condition (\ref{conv_D}) if $j = 1$). We  construct a continuous map $f_{j+1} : \overline{\D} \to \C^n$ that is holomorphic on $\D$, a partition $(E_i^{({j+1})})_{i=1}^{k_{j+1}}$ of $E$ into pairwise disjoint closed subsets, and points $(a_i^{({j+1})})_{i=1}^{k_{j+1}}$ with $a_i^{({j+1})} \in E_i^{({j+1})}$ for $i = 1,\dots,k_{j+1}$, satisfying conditions (\ref{conv_E})--(\ref{control_image}).
		
	By the uniform continuity of $f_j$ on $\overline{\D}$ there exists $\delta_{j+1} > 0$ such that, for $x, y \in \overline{\D}$,
	\[\lvert x - y \rvert < \delta_{j+1} \implies \lvert f_j(x) - f_j(y) \rvert < \epsilon/2^{j+2}\,.\]
	Shrinking $\delta_{j+1}$ further if necessary, it follows by the uniform continuity of $\varphi$ on $E$ that, for $x, y \in E$,
	\[\lvert x - y \rvert < \delta_{j+1} \implies \lvert \varphi(x) - \varphi(y) \rvert < \eta_{j+1}\,.\]
	
	Partitioning each $E^{(j)}_i$, $i = 1, \dots, k_j$, into sufficiently small closed subsets gives a collection of pairwise disjoint sets $E^{({j+1})}_i$, $i = 1, \dots, k_{j+1}$, satisfying $\bigcup_{i = 1}^{k_{j+1}} E^{({j+1})}_i = \bigcup_{i' = 1}^{k_{j}} E^{({j})}_{i'} = E$, $\diam(E^{({j+1})}_i) < \delta_{j+1}$ for $i = 1,\dots,k_{j+1}$, and such that for each $i \in \{1,\dots,k_{j+1}\}$ there exists a unique value $i' \in \{1, \dots, k_j\}$ for which $E^{({j+1})}_i \subset E^{(j)}_{i'}$. For each $i = 1,\dots,k_{j+1}$ we choose a point $a^{({j+1})}_i \in E^{({j+1})}_i$. Define $\alpha_{j+1} : E \to \partial \D$ by $\alpha_{j+1}(E^{({j+1})}_i) = a^{({j+1})}_i$ for $i = 1,\dots,k_{j+1}$. By Theorem \ref{thm_RudinCarleson}, $\alpha_{j+1}$ extends to a continuous map $\alpha_{j+1} : \overline{\D} \to \overline{\D}$, holomorphic on $\D$, such that $\lVert \alpha_{j+1} - \id \rVert_{\overline{\D}} < \delta_{j+1}$, $\alpha_{j+1}(0) = 0$, and $\alpha_{j+1}(\D) \subset \D$.
	
	Given $i \in \{1,\dots,k_{j+1}\}$, let $i' \in \{1,\dots,k_j\}$ be the unique value such that $E^{({j+1})}_i \subset E^{(j)}_{i'}$. By condition (\ref{image_E}) it follows that $f_j(a^{({j+1})}_i) = \varphi(a^{(j)}_{i'})$. By condition (\ref{conv_E}) we also have that $\lvert f_j(a^{({j+1})}_i) - \varphi(a^{({j+1})}_i) \rvert < \eta_j$. Choose an arc from the point $f_j(a^{({j+1})}_i) = \varphi(a^{(j)}_{i'})$ to $\varphi(a^{({j+1})}_i)$ that is contained in the open ball $B(\varphi(a^{(j)}_{i'});\eta_j) \subset V$ of radius $\eta_j$ centred at $\varphi(a^{(j)}_{i'})$. Choose a neighbourhood $U^{({j+1})}_i$ of $a^{({j+1})}_i$ in $\overline{\D}$ such that $f_j(\overline{U^{({j+1})}_i}) \subset B(\varphi(a^{(j)}_{i'});\eta_j)$.
	
	Perform the preceding construction for each $i = 1, \dots, k_{j+1}$. We now apply Theorem \ref{thm_ForstnericWold} to $f_j$, with $V_i = B(\varphi(a^{(j)}_{i'});\eta_j)$, $i = 1,\dots,k_{j+1}$ (here, $i'$ depends on $i$ as described above). We thereby obtain a continuous map $g_{j+1} : \overline{\D} \to \C^n$, holomorphic on $\D$, such that 
	\begin{itemize}
		\item $\lVert f_j - g_{j+1} \rVert_{\overline{\D} \setminus \bigcup_{i = 1}^{k_{j+1}} U^{({j+1})}_i} < \epsilon / 2^{j+2}$.
		\item $g_{j+1}(a^{({j+1})}_i) = \varphi(a^{({j+1})}_i)$ for $i = 1,\dots,k_{j+1}$.
		\item $g_{j+1}(0) = f_j(0)$.
		\item $g_{j+1}(U^{({j+1})}_i) \subset V_i = B(\varphi(a^{(j)}_{i'});\eta_j)$ for $i = 1,\dots,k_{j+1}$, where, for each $i$, we let $i' \in \{1,\dots,k_j\}$ be the unique value such that $E^{({j+1})}_i \subset E^{(j)}_{i'}$.
		\item $g_{j+1}(\overline{\D}) \subset V$ (here we use the fact from the beginning of the proof that\\ $B(\varphi(a^{(j)}_{i'});\eta_j) \subset B(\varphi(a^{(j)}_{i'});\eta_1) \subset V$).
	\end{itemize}
	
	Let $z \in \overline{\D}$. If $z \in {\overline{\D} \setminus \bigcup_{i =1}^{k_{j+1}} U^{({j+1})}_i}$ then $\lvert f_j(z)-g_{j+1}(z)\rvert < \epsilon/2^{j+2}$ from the first property of $g_{j+1}$ above. Otherwise, suppose $z \in U^{({j+1})}_i$ for some $i = 1, \dots, k_{j+1}$. Then, for the appropriate value of $i'$ we have that $f_j(z) \in f_j(U^{({j+1})}_i) \subset B(\varphi(a^{(j)}_{i'});\eta_j)$ and $g_{j+1}(z) \in g_{j+1}(U^{({j+1})}_i) \subset B(\varphi(a^{(j)}_{i'});\eta_j)$. Thus, $\lvert f_j(z) - g_{j+1}(z) \rvert < \diam(B(\varphi(a^{(j)}_{i'});\eta_j)) = 2\eta_j = \epsilon/2^{j+2}$. We therefore see that $\lVert f_j - g_{j+1} \rVert_{\overline{\D}} < \epsilon/2^{j+2}$.
	
	Let $f_{j+1} = g_{j+1} \circ \alpha_{j+1}$. We now show that $f_{j+1}$ has all required properties. Clearly, $f_{j+1}$ is continuous on $\overline{\D}$, holomorphic on $\D$, and satisfies both $f_{j+1}(\overline{\D}) \subset V$ and $f_{j+1}(0) = f_j(0) = f(0)$. Let $z \in E$, so that $z \in E^{({j+1})}_i$ for a unique value $i \in \{1,\dots,k_{j+1}\}$. We have $f_{j+1}(z) = g_{j+1}(\alpha_{j+1}(z)) = g_{j+1}(a^{({j+1})}_i) = \varphi(a^{({j+1})}_i)$, and therefore
	\begin{equation*}
		\lvert f_{j+1}(z) - \varphi(z)\rvert  = \lvert \varphi(a^{({j+1})}_i) - \varphi(z) \rvert < \eta_{j+1}
	\end{equation*}
	since $a^{({j+1})}_i, z \in E^{({j+1})}_i$ and $\diam(E^{({j+1})}_i) < \delta_{j+1}$. This shows that $\lVert f_{j+1} - \varphi \rVert_E < \eta_{j+1}$.
	
	Finally, for $z \in \overline{\D}$,
	\begin{eqnarray*}
		\lvert f_{j+1}(z) - f_j(z) \rvert &=& \lvert g_{j+1}(\alpha_{j+1}(z)) - f_j(z) \rvert\\
		&\le& \lvert g_{j+1}(\alpha_{j+1}(z)) - f_j(\alpha_{j+1}(z)) \rvert + \lvert f_j(\alpha_{j+1}(z))- f_j(z) \rvert\\
		& < & \epsilon/2^{j+2} + \epsilon/2^{j+2} = \epsilon/2^{j+1}
	\end{eqnarray*}		
	since $\lvert \alpha_{j+1}(z) - z \rvert < \delta_{j+1}$. Thus, $\lVert f_{j+1} - f_j \rVert _{\overline{\D}} < \epsilon/2^{j+1}$. The same estimate holds on $K \subset \overline{\D}$.
\end{proof}

\section{Main result for general complex manifold targets}

\noindent
We now give the main result of our paper, generalising Theorem \ref{thm_main_Cn} to the case where the target $\C^n$ is replaced by an arbitrary connected complex manifold $X$.

\begin{theorem}
	\label{thm_main_generaltarget}
	Let $X$ be a connected complex manifold equipped with a distance function induced by a complete Riemannian metric.	Let $f : \overline \D \to X$ be a continuous map that is holomorphic on $\D$. Let $E \subset \partial \D$ be a closed set of measure zero and let $\varphi : E \to X$ be a continuous map. Let $c \in \D$. Then, $f$ can be uniformly approximated on compact subsets of $\overline \D \setminus E$ by continuous maps $F:\overline \D \to X$ that are holomorphic on $\D$, such that $F\vert_E = \varphi$ and $F(c) = f(c)$. Furthermore, if $U \subset X$ is any connected neighbourhood of $f(\overline \D) \cup \varphi(E)$ then it is possible to ensure that $F(\overline \D) \subset U$.
\end{theorem}

The proof of Theorem \ref{thm_main_generaltarget} follows by the same argument as given earlier for Theorem \ref{thm_main_Cn}, replacing the use of Theorem \ref{thm_ForstnericWold} on attaching fingers to maps $\overline \D \to \C^n$ by the following generalisation for maps $\overline \D \to X$.


\begin{theorem}
	\label{thm_ForstnericWold_generaltarget}
	Let $X$ be a connected complex manifold equipped with a distance function defined by a complete Riemannian metric.
	Let $f : \overline \D \to X$ be a continuous map that is holomorphic on $\D$. Let $a_1, \dots, a_k \in \partial \D$ be distinct points. For $j = 1, \dots, k$, let $U_j \subset \overline \D$ be a (relatively open) neighbourhood of $a_j$ such that $U_1,\dots,U_k$ are pairwise disjoint. Let $b_1, \dots, b_k \in X$. For $j = 1, \dots, k$, let $\lambda_j \subset X$ be a continuous arc from $f(a_j)$ to $b_j$. Let $c_1, \dots, c_s \in \overline \D \setminus \{a_1,\dots,a_k\}$. For $j = 1, \dots, k$, let $V_j \subset X$ be a neighbourhood of $f(\overline U_j) \cup \lambda_j$, and let $V \subset X$ be a neighbourhood of $f(\overline \D) \cup \bigcup_{j=1}^k \lambda_j$. Then, $f$ can be uniformly approximated on $\overline \D \setminus \bigcup_{j=1}^k U_j$ by continuous functions $g : \overline \D \to X$ that are holomorphic on $\D$, such that $g(a_j) = b_j$ for $j = 1 , \dots, k$, $g(U_j) \subset V_j$ for $j = 1 , \dots, k$,
	$g(c_i) = f(c_i)$ for $i = 1,\dots,s$, and $g(\overline \D) \subset V$.
\end{theorem}

\begin{proof}
	Define a continuous extension $f' : K \to X$ of $f$ as in the proof of Theorem \ref{thm_ForstnericWold}. By a result of Poletsky \cite{Poletsky2013} (see also the proof of Theorem 1.4 in \cite{Forstneric2019}), the graph $\Gamma(f')  = \{(z,f'(z)):z \in K\}\subset \C \times X$ has a basis of Stein neighbourhoods in $\C \times X$. Let $T$ be one such Stein neighbourhood, and let $\Psi : T \to \C^N$ be a proper holomorphic embedding, for some $N \in \N$. Let $W$ be a tubular neighbourhood of $\Psi(T)$ with a holomorphic retraction $\kappa : W \to \Psi(T)$.
	We apply Mergelyan's theorem with interpolation at the points $ra_1,\dots,ra_k, c_1,\dots,c_s$ to the map $z \mapsto \Psi(z, f'(z))$, giving an entire map $\sigma : \C \to \C^N$. If the approximation is sufficiently good then $\sigma(K) \subset W$ and restriction to a sufficiently small relatively compact neighbourhood $U$ of $K$ ensures that $\sigma(U) \subset W$. Applying $\kappa$ followed by $\Psi^{-1}$ and then the projection $\pi_X : \C \times X \to X$ yields the holomorphic map $h = \pi_X \circ \Psi^{-1} \circ \kappa \circ \sigma : U \to X$ which has all of the necessary properties (assuming the approximation was sufficiently close) for the rest of the proof to be completed as in Theorem \ref{thm_ForstnericWold}.
\end{proof}

\end{document}